\documentclass[12pt,reqno]{amsart}
\usepackage{amsmath,amssymb,amsfonts,amscd,latexsym,amsthm,mathrsfs,verbatim, hyperref}
\usepackage{graphicx}
\newtheorem{lemma}{Lemma}
\newtheorem{theorem}{Theorem}
\newtheorem{corollary}{Corollary}

\theoremstyle{remark}
\newtheorem*{remark}{\bf Remark}
\newtheorem*{acknowledgements}{\bf Acknowledgements}
\newcommand\sump{\sideset{}{'}\sum}

\renewcommand\Im{\operatorname{Im}}
\renewcommand\d{\,\mathrm d}

\newcommand{\Eis}{\operatorname{Eis}}
\newcommand{\Eistilde}{\widetilde{\operatorname{Eis}}}
\newcommand{\cF}{\mathcal F}
\newcommand{\cE}{\mathcal E}
\newcommand{\cn}{\operatorname{cn}}
\newcommand{\sn}{\operatorname{sn}}
\newcommand{\nd}{\operatorname{nd}}
\newcommand{\dn}{\operatorname{dn}}
\newcommand{\sd}{\operatorname{sd}}
\newcommand{\nc}{\operatorname{nc}}
\newcommand{\cd}{\operatorname{cd}}
\let\kp\kappa
\let\vt\vartheta
\begin{document}
\title{$L$-values for conductor $32$}
\author{Boaz Moerman}
\address{Faculty of Science, Radboud University, Nijmegen, NETHERLANDS}
\email{boazmoerman@gmail.com}

\date{\today}
\begin{abstract}
In recent years, Rogers and Zudilin developed a method to write $L$-values attached to elliptic curves as periods. In order to apply this method to a broader collection of $L$-values, we study Eisenstein series and determine their Fourier series at cusps. Subsequently, we write the $L$-values of an elliptic curve of conductor 32 as an integral of Eisenstein series and evaluate the value at $k>1$ explicitly as a period. As a side result, we give simple integral expressions for the generating functions of $L(E,k)$ when even (or odd) $k$ runs over positive integers.
\end{abstract}
\maketitle


\tableofcontents


\section{Introduction}
A \emph{period} is a complex number whose real and imaginary parts are both (absolutely convergent) integrals of rational functions with rational coefficients over domains in $\mathbb{R}^n$ defined by polynomial inequalities with rational coefficients \cite{KoZa01}. The rationality of functions can be relaxed to their algebraicity, without affecting the definition (apart from possibly reducing the dimension $n$ of the domain of integration).
The set of periods $\mathscr{P}$ forms a countable ring, which contains the algebraic numbers $\overline{\mathbb{Q}}$.
The ring includes important mathematical constants such as integer values of the Riemann zeta function for $k>1$,
$$\zeta(k)=\sum_{n=1}^{\infty}\frac{1}{n^k}=\idotsint\limits_{[0,1]^k}\frac{\d x_1 \d x_2\dotsb \d x_k}{1-x_1 x_2\dotsb x_k}.$$

The extended period ring $\widehat{\mathscr{P}}:=\mathscr{P}[1/\pi]$ contains a large collection of natural examples, such as values of generalized hypergeometric functions \cite{Sl66} at algebraic points and it is conjectured to contain special $L$-values.
For instance, theorems by Beilinson and Deninger--Scholl state that the (non-critical) value of the $L$-series attached to a cusp form $f(\tau)$ of weight $k$ (with algebraic coefficients) at a positive integer
$m \geq k$ (see the definitions and formula \eqref{L-function} below) belongs to $\widehat{\mathscr{P}}$. Although the proof of the theorems is effective, computing these $L$-values explicitly as periods remains a very tough
problem even in particular cases. A large portion of these computations is
motivated by evaluations of the logarithmic Mahler measures
$$\frac{1}{(2\pi i)^k}\idotsint\limits_{|x_1|=\dots=|x_k|=1}\log{|P(x_1,\dots,x_k)|}\,\frac{\d x_1}{x_1}\dotsb\frac{\d x_k}{x_k}$$
of multivariate polynomials $P\in \mathbb{Z}[x_1,\dots,x_k]$ as $L$-values, where the latter integral itself is transparently a period.
The existence of such evaluations is only conjectural in most cases, and some of these are deducible from Beilinson's conjectures. Some explicit results on regulator integrals were proven by Zudilin \cite{Zu14} and Brunault \cite{Br16}, and these results were considerably generalized last year by the thesis of Wang \cite{Wa20}. 
For an introduction to Mahler measures, related open problems and known results, see the book \cite{BrZu20}.

With the purpose of proving some of the conjectures on Mahler measures in the case $k=2$, Rogers and Zudilin developed a framework \cite{RoZu12, RoZu13} for writing the $L$-values $L(f,2)$ of cusp forms $f(\tau)$ of weight $2$ as periods. Zudilin \cite{Zu13} later described an algorithm behind their method, which relates the $L$-values of a cusp form $f$ to $L$-values of a different modular-like object; see also \cite{Ro13}. This can be used in some cases to compute values $L(f,k)$ of the modular form as periods. However, executing this method explicitly remains a difficult task.

Throughout this paper, we use the notation $q=e^{2\pi i \tau}$ for $\tau$ in the upper half-plane $\Im \tau>0$, so that $|q|<1$.
For functions of variable $\tau$ or $q$, we introduce the differential operator
$$\delta=\frac{1}{2\pi i}\,\frac{\d}{\d \tau}=q\frac{\d}{\d q}$$
and denote by $\delta^{-1}$ the associated anti-derivative normalized by 0 at $\tau=i\infty$ (or at $q=0$):
$$\delta^{-1}f=\int_{0}^{q}f\frac{\d q}{q}.$$
Then for a modular form $f(\tau)=\sum_{n=1}^{\infty}a_n q^n$ whose expansion vanishes at infinity, we have
\begin{align}
L(f,k)&:=\frac{1}{(k-1)!}\int_{0}^{1}f\log^{k-1}q\,\frac{\d q}{q}=\frac{(2\pi)^k}{(k-1)!}\int_{0}^{\infty}f(it)t^{k-1}\d t,
\label{L-function}
\end{align}
in particular,
$$L(f,k)=\sum_{n=1}^{\infty}\frac{a_n}{n^k}=(\delta^{-k}f)|_{q=1}$$
whenever the sum makes sense.

One of the themes of this paper is a study of the Fourier series at cusps of Eisenstein series of arbitrary level. They are then used to help with executing the method of Rogers and Zudilin \cite{Ro13,RoZu12, RoZu13, Zu13}.
As a particular illustration of the ideas in this paper, the $L$-values of the elliptic curve
$$E:y^2=x^3-x$$
defined over $\mathbb{Q}$ are studied. This curve has conductor $32$;
by the modularity theorem, we can associate with $E$ the unique newform
$$f(\tau)=q\prod_{m=1}^\infty(1-q^{4m})^2(1-q^{8m})^2=q-2q^5-3q^9+ 6q^{13} + 2q^{17} +\dotsb$$
of weight $2$ and level $32$, so that $L(E,s)=L(f,s)$, and the equality remains valid for \emph{any} choice of conductor~32 elliptic curve over~$\mathbb{Q}$.

Our first principal result gives an example of evaluating $L(E,4)=L(f,4)$ as a period, namely as an integral of a simple algebraic expression over the $4$-dimensional cube, a task that was never explicitly done before.

\begin{theorem} \label{L(E,4)}
The $L$-value of an elliptic curve $E$ of conductor $32$ at $4$ possesses the following expression as a period:
$$L(E,4)=\frac{\pi^3}{1536}\iiiint\limits_{[0,1]^4}\frac{(1-6y+y^2)\sqrt{1-y} \,\d y \,\d y_1 \,\d y_2 \,\d y_3}{\sqrt{y(1+y)}\big(y^4+4(1-y^2)(1-y_1^2)(1-y_2^2)(1-y_3^2)\big)}.$$
\end{theorem}

This expression is reminiscent of the period representations for $L(E,2)$ and $L(E,3)$ given in \cite{Zu13}.
Furthermore, we show how our results on Eisenstein series can be used to write a general $L$-value $L(E,k)$ of the same curve as an integral of products of two Eisenstein series.
Then we convert the resulting representation to a period expression for $L(E,k)$.
To state our final results in a compact form, we introduce a particular case of the Euler--Gauss hypergeometric function:
\begin{equation}
\cF(u)=\frac{1}{\operatorname{agm}(1,\sqrt{u})},
\label{cF-def}
\end{equation}
where $\operatorname{agm}(x,y)$ is the algebraic-geometric mean of $x$ and $y$.
This function can be alternatively given as a period, namely for real $u$,
\begin{equation}
\cF(u)=\begin{cases}
\displaystyle\frac1\pi\int_0^1t^{-1/2}(1-t)^{-1/2}(1-ut)^{-1/2}\,\d t, & \text{if } u<1, \\[3mm]
\displaystyle\frac1\pi\int_0^1\frac{\d t}{\sqrt{t(1-t)}}\biggl(\frac1{\sqrt{u-t}}-\frac i{\sqrt{u-(u-1)t}}\biggr), &\text{if } u>1.
\end{cases}
\label{cF-alt}
\end{equation}

We define the basic Jacobi elliptic functions $\sn$, $\cn$ and $\dn$ as the solutions of the differential equations 
\begin{equation}
\begin{aligned}
\frac{\d \sn(u,\kp)}{\d u}&=\cn(u,\kp)\dn(u,\kp), \\
\frac{\d \cn(u,\kp)}{\d u}&=-\sn(u,\kp)\dn(u,\kp), \\
\frac{\d \dn(u,\kp)}{\d u}&=-\kp^2\sn(u,\kp)\cn(u,\kp), \\
\end{aligned}
\label{eq-A}
\end{equation}
with the Maclaurin expansions
\begin{equation}
\begin{aligned}
\sn(u,\kp)&=u-\frac{1}{6}(1+\kp^2)u^3+\mathcal{O}(u^5), \\
\cn(u,\kp)&=1-\frac{1}{2}u^2+\mathcal{O}(u^4), \\
\dn(u,\kp)&=1-\frac{1}{2}\kp^2 u^2+\mathcal{O}(u^4).
\end{aligned}
\label{eq-B}
\end{equation}
Now the period representations for $L(E,k)$ can be read off from the following generating functions.

\begin{theorem}
\label{th2}
We have
\begin{align*}
\sum_{\substack{k\geq 1 \\ k\;\text{even}}}\frac{L(E,k)}{k-1}x^{k-1}
&=\frac{\pi}{16}\int_0^1\!\!\frac{\d\alpha}{\sqrt{1-\alpha^2}}\int_0^{2\sqrt{\alpha}/(1-\alpha)}\frac{\d\beta}{\sqrt{1+\beta^2}}
\\ & \quad
\times\frac{1}{C}\int_0^{1}\sd\big(e^{2\pi it}\sqrt{Cx},\sqrt{1-\alpha^2}\big)\sd\big(e^{-2\pi it}\sqrt{Cx},\sqrt{1+\beta^2}\big)\,\d t,
\end{align*}
where
$$C=C(\alpha,\beta)=\frac{\pi^2 i}{4}\bigl(\cF(\alpha^2)\cF(1+\beta^2)-\tfrac{1}{2} \cF(1-\alpha^2)\cF(-\beta^2)+i\cF(\alpha^2)\cF(-\beta^2)\bigr)$$
and $\sd$ is the Jacobi elliptic function $\sd=\frac{\sn}{\dn}$.
\end{theorem}

\begin{theorem}
\label{th3}
We have
\begin{align*}
&
\sum_{\substack{k>1 \\k\;\text{odd}}} L(E,k)x^{k-1}
=-\frac{\pi i}{32}\int_0^1\!\frac{\d \alpha}{\alpha\sqrt{1-\alpha}}\int_{p(\alpha)}^\infty \frac{\d \beta}{\sqrt{\beta}(1-\beta)}
\\ & \qquad
\times\frac{1}{C}\int_0^1\big(\cd\big(\tfrac{1}{2}e^{2\pi i t}\sqrt{Cx},\sqrt{1-\alpha}\big)\nc\big(e^{-2\pi i t}\sqrt{Cx},\sqrt{1-\beta}\big)-1\big)\d t,
\end{align*}
where
\begin{gather*}
p(\alpha)=\frac{-8 \alpha+4\sqrt{4 \alpha^2+\alpha (1-\alpha)^2}}{(1-\alpha)^2},\\
C=C(\alpha,\beta)=\frac{\pi^2 i}{2}\bigl(\cF(\alpha)\cF(1-\beta)-\tfrac{1}{2} \cF(1-\alpha)\cF(\beta)\bigr),
\end{gather*}
and $\cd$ and $\nc$ are the Jacobi elliptic functions $\cd=\frac{\cn}{\dn}$ and $\nc=\frac1{\cn}$.
\end{theorem}

These theorems are both direct consequences of Theorem \ref{period integrals} below, which gives an explicit period representation for $L(E,k)$.

In what follows, we occasionally use the Sturm bound \cite[Corollary 9.20]{St07} to establish an equality between two different modular forms of the same weight through verifying the equality of the first terms in their Fourier series. If a series $F(q)\in \mathbb{Z}[[q]]$ satisfies $F(0)=1$, then we see (by using the M\"obius inversion) that it can be expressed as an infinite product $\prod_{n=1}^\infty (1-q^n)^{a(n)}$ by means of the formula

$$a(n)=-\frac{1}{n}\sum_{l\mid n}\mu\bigg(\frac{n}{l}\bigg)b(l),$$

where $b(l)$ is the $l$-th coefficient of the logarithmic derivative of $F(q)$. Combining this fact with the Sturm bound gives an easy method to identify modular functions as eta quotients.

We will also use the generalized hypergeometric functions, which are defined by the series
$$_{k+1}F_{k}\biggl(\begin{matrix} a_0, \, a_1,\dots, a_k \\ b_1,\dots, b_k \end{matrix}\biggm|z\biggr)=\sum_{n=0}^{\infty}\frac{(a_0)_n(a_1)_n\dots(a_k)_n}{(b_1)_n\dots(b_k)_n}\frac{z^n}{n!}$$
in the disk $|z|<1$; here $(a)_n:=\Gamma(a+n)/\Gamma(a)=\prod_{m=0}^{n-1}(a+m)$ denotes the Pochhammer symbol. The properties of the series, including integral representations and analytic continuation, can be found in Slater's book \cite[Chap.~4]{Sl66}. With this notation we have $$\cF(u)={}_2F_1(\tfrac{1}{2},\tfrac{1}{2};1;u)={}_2F_1\biggl(\begin{matrix} \tfrac{1}{2}, \, \tfrac{1}{2} \\ 1 \end{matrix}\biggm|u\biggr)$$ for the special case in~\eqref{cF-def},~\eqref{cF-alt}.

\begin{acknowledgements}
This paper grew up from my bachelor thesis defended in Summer 2019.
I owe a great debt to all who have aided me in my thesis and in my study. In particular, I am deeply grateful for my advisor Wadim Zudilin, who has taught me many things and advised me on several matters, such as how to write a thesis and how to do mathematical research. He also introduced me to several interesting topics, among which is the topic of this thesis.
I would also like to thank the Radboud University for allowing me to take part in their Honours programme for bachelor students. It allowed me to have research visits abroad. One of these was at the \'ENS Lyon, where I had advice from Fran\c{c}ois Brunault and Riccardo Pengo, and also learned a lot about algebraic geometry and its connection to my thesis. The other was at the TU Darmstadt, where Michalis Neururer gave more advice and taught me more about modular forms and \texttt{SageMath}. I am very grateful for their time and wisdom, both at a mathematical and personal level. I also want to thank Walter Van Assche for the references \cite{Ro07,St94}, which clarified the relation between our results and Duke's work \cite{Du08}.
\end{acknowledgements}

\section{Eisenstein series and their expansions at cusps} \label{Eisenstein series}
\label{s1}
Schoeneberg in \cite[Chapter 7]{Sch74} considers the following Eisenstein series of level $N$ and weight $k$ and describes their transformation laws, if $k>2$:
\begin{gather}
G_{N,k,(a,b)}(\tau)=\sump_{\substack{m\equiv a \\ n\equiv b} \bmod N} (m\tau+n)^{-k},
\label{G}
\end{gather}
where the dash means that $(m,n)=(0,0)$ is excluded from summation.
When $k=1$ or $2$, it is defined similarly, though analytic continuation is required to circumvent the lack of absolute convergence. These Eisenstein series are modular forms for the principal congruence subgroup $\Gamma(N)$.

We will work with a slight modification of these series and use the transformation laws for it to find two different expansions of the series at their cusps.

In order to do this, we take integers $N$, $k$, $a$ and $b$, where $k$ is nonnegative and $N$ is positive. Before introducing the version of the Eisenstein series used here, we need to define their constant terms:
\begin{equation*}
\gamma_{a,b}(\tau)=
\begin{cases}
\beta_k \alpha^{N,k}_{a,b} &\text{ if } k\neq 2
\\ \beta_2 \big(\alpha^{N,2}_{a,b}-\frac{2\pi i}{N^2(N\tau - N\overline{\tau})}\big) &\text{ if } k=2;
\end{cases}
\end{equation*}
where $$\beta_{k}=\frac{(k-1)!}{(-2\pi i)^k},$$
and $\alpha^{N,k}_{a,b}$ is defined by setting
$$\alpha_{a,b}^{N,k}=0 \text{ if }a\not\equiv 0\bmod N$$
and otherwise, when $k>1$:
$$\alpha_{a,b}^{N,k}=
\sump_{\substack{m\in \mathbb{Z} \\ m\equiv b\bmod N}}m^{-k}=\frac{1}{N^k}\left(\zeta(k,\{\tfrac{b}{N}\})+(-1)^k\zeta(k,-\{\tfrac{b}{N}\})\right),$$
and when $k=1$:
\begin{align*}
\alpha_{a,b}^{N,k}=&\frac{1}{N}\lim_{s\to 0} \left(\zeta(1+s,\{\tfrac{b}{N}\})-\zeta(1+s,-\{\tfrac{b}{N}\})\right)
\\ & \quad -\frac{\pi i}{N}\left(\zeta(0,\{\tfrac{a}{N}\})-\zeta(0,-\{\tfrac{a}{N}\})\right),
\end{align*}
where $\{x\}$ denotes the fractional part of $x$ and $\zeta(s,t)$ denotes the Hurwitz zeta function \cite[\S\,13.11]{WhWa27} which is defined as the analytic continuation of the series $\sum_{n=0}^\infty (n+t)^{-s}$.
Now we define the Eisenstein series $E_{a,b}=E_{a,b}^{N,k}$ of level $N$ and weight $k$ as
\begin{equation}
E_{a,b}(\tau)=\gamma_{a,b}(\tau)+\sum_{\substack{n,m\geq 1 \\ m\equiv a \bmod N}}\zeta_N^{bn}n^{k-1}q^{mn}
+(-1)^{k}\sum_{\substack{n,m\geq 1 \\ m\equiv -a \bmod N}}\zeta_N^{-bn}n^{k-1}q^{mn},
\label{E}
\end{equation}
where $\zeta_N=e^{\frac{2 \pi i}{N}}$.

With this definition we will later see that 
\begin{equation}
E^{N,k}_{a,b}(\tau)=\beta_{k} G_{N,k,(a,b)}(N\tau) \text{ for any }k>0,
\label{eG}
\end{equation}
so that the Eisenstein series is a modular form for $\Gamma_1(N^2)$ for $k\neq 2$, and is quasimodular for $k=2$.

In order to express the expansions of $E_{a,b}$ around $\frac{c}{N}$, where $c$ is an integer, define $$E^{N,k}_{a,b,c}(\tau)=E_{a,b,c}(\tau)=E_{a,b}\left(\frac{c}{N}+\tau\right).$$
The following theorem gives two Fourier expansions for this general Eisenstein series.

\begin{theorem}
\label{th4}
For $a, b, c$ arbitrary integers, $E_{a,b,c}=E^{N,k}_{a,b,c}$ possesses the following expansions:
\begin{align}
E_{a,b,c}(\tau)
&=E_{a,-a'}(\tau)+\gamma_{a,b}(\tau)-\gamma_{a,-a'}(\tau)
\label{fourier1}
\end{align}
and
\begin{align}
E_{a,b,c}(\tau)(N\tau)^{k} &=E_{a',a}\bigg(\frac{-1}{N^2\tau}\bigg)+\gamma_{a',b'}\bigg(\frac{-1}{N^2\tau}\bigg)-\gamma_{a',a}\bigg(\frac{-1}{N^2\tau}\bigg),
\label{fourier2}
\end{align}
where $a'=-ac-b$ and $b'=a(c^2+1)+bc$
\textup{(see equation \eqref{a dash} below).}
\end{theorem}
If we take $c$ to be zero, we obtain the following corollary.
\begin{corollary}
	For any integers $a$ and $b$, $E_{a,b}=E^{N,k}_{a,b}$ satisfies
\begin{equation}
	E_{a,b}(\tau)(N\tau)^k=E_{-b,a}\bigg(\frac{-1}{N^2\tau}\bigg).
	\label{mod involution}
\end{equation}
\end{corollary}
Below we make occasional use of the identity
$$E^{N,k}_{-a,-b}(\tau)=(-1)^k E^{N,k}_{a,b}(\tau),$$
which is an immediate consequence of the definition of $E^{N,k}_{a,b}$.

\begin{remark} One can in fact generalize the definition of Eisenstein series and the above theorem about the expansions to the case $k=0$. The Eisenstein series we then obtain are the logarithms of generalized Dedekind eta functions
\begin{align*}
g_{a,b}(\tau)=q^{B(a/N)/2}\prod_{\substack{m\ge1\\m\equiv a\bmod N}}(1-\zeta_N^{b}q^{m})\prod_{\substack{m\ge1\\m\equiv -a\bmod N}}(1-\zeta_N^{-b}q^{m}),
\label{g}
\\
\text{ where } B(x)=\{x\}^2-\{x\}+1/6,
\end{align*}
and by applying \cite[Theorem 1]{Ya04} with $\gamma=\left(\begin{smallmatrix} c & -c^2-1 \\ 1 & -c \end{smallmatrix}\right)$ implies that similar formulas to \eqref{fourier1} and \eqref{fourier2} hold. The only difference is that there is a contribution of a constant term on the right-hand sides of \eqref{fourier1} and \eqref{fourier2} which has to be accounted for.
\end{remark}

In \cite[Chapter 7]{Sch74}, the following expansions for $G_{N,k,(a,b)}$ in \eqref{G} are given:
\begin{align*}
G_{N,k,(a,b)}(\tau)
&=\alpha^{N,k}_{a,b}+\frac{1}{\beta_{k}} \sum_{m\equiv a \bmod N} \sum_{nm>0}n^{k-1}\cdot \text{sgn} \: n\cdot e^{\frac{2\pi i}N (bn+\tau nm)}
\\ &\quad
-\delta_{k,2}\frac{2\pi i}{N^2(\tau - \overline{\tau})}
\\ &
=\alpha^{N,k}_{a,b}
\\ &\quad
+\frac{1}{\beta_{k}}\bigg(\sum_{\substack{m,n\ge1\\m\equiv a \bmod N}}n^{k-1} e^{\frac{2\pi i}N (bn+\tau nm)}
+(-1)^k\sum_{\substack{m,n\ge1\\m\equiv -a \bmod N}}n^{k-1} e^{\frac{2\pi i}N (-bn+\tau nm)} \bigg)
\\ &\quad
-\delta_{k,2}\frac{2\pi i}{N^2(\tau - \overline{\tau})},
\end{align*}
where $\delta$ is the Kronecker delta function.
Thus, $E_{a,b}(\tau)=\beta_k G_{N,k,(a,b)}(N\tau)$ as previously asserted in \eqref{eG}.

\begin{proof}
We now derive Fourier expansions of $E_{a,b,c}^{N,k}(\tau)$ for $k>0$, in terms of $\tau$ and in terms of $-\frac{1}{N^2\tau}$.

\medskip
\noindent
(a) \textsl{Expansion in $\tau$}.
The Fourier expansion \eqref{fourier1} is obtained simply by writing out the definition of $E_{a,b}$ if $k\neq 2$:
\begin{align*}
E_{a,b,c}^{N,k}(\tau)
&=\beta_{k}\alpha^{N,k}_{a,b}
\\ &\quad
+\bigg(\sum_{\substack{m,n\ge1\\m\equiv a \bmod N}}n^{k-1} e^{\frac{2\pi i}N (bn+(c+N\tau) nm)}
+(-1)^k\sum_{\substack{m,n\ge1\\m\equiv -a \bmod N}}n^{k-1} e^{\frac{2\pi i}N (-bn+(c+N\tau) nm)} \bigg)
\\ &
=\beta_{k}\alpha^{N,k}_{a,b}
\\ &\quad
+\bigg(\sum_{\substack{m,n\ge1\\m\equiv a \bmod N}}n^{k-1} \zeta_{N}^{bn+acn}q^{mn}
+(-1)^k\sum_{\substack{m,n\ge1\\m\equiv -a \bmod N}}n^{k-1} \zeta_{N}^{-bn-acn}q^{mn}\bigg).
\end{align*}
When $k=2$ we have the same expression, with the extra term
$$-\beta_{2}\frac{2 \pi i}{N^2((c+N\tau)-(\overline{c+N\tau}))}=-\beta_{2}\frac{2 \pi i}{N^2 (N\tau -N\overline{\tau})}=-\beta_{2}\frac{\pi}{N^3 \Im(\tau)}$$
included. This establishes \eqref{fourier1} for any $k>0$.

\medskip
\noindent
(b) \textsl{Expansion in $\frac{-1}{N^2\tau}$}.
Now we will derive an Fourier expansion in terms of $\frac{1}{N^2\tau}$.
For every $B=\bigl(\begin{smallmatrix} b_{11} & b_{12} \\ b_{21} & b_{22} \end{smallmatrix}\bigr) \in \Gamma$, we have $$G_{N,k,(a,b)}(B\tau)=(b_{21}\tau+b_{22})^{k}G_{N,k,(a,b)B^{t}}(\tau)$$
(see \cite[Chapter 7]{Sch74}).
Now if we define
\begin{equation}
A=\left(\begin{matrix} c & -c^2-1 \\ 1 & -c \end{matrix}\right)
\label{gamma},
\end{equation} so that
\begin{equation}
(\begin{matrix} a' & b' \end{matrix})=(\begin{matrix} a & b \end{matrix})\left(\begin{matrix} c & -c^2-1 \\ 1 & -c \end{matrix}\right)^{-1}
=(\begin{matrix} -ac-b & a(c^2+1)+bc \end{matrix}), \label{a dash}
\end{equation}
then
 $$G_{N,k,(a',b')}\left(\frac{c\tau-(c^2+1)}{\tau-c}\right)=G_{N,k,(a',b')}(A\tau)=G_{N,k,(a,b)}(\tau)(\tau-c)^{k}.$$
Substituting $c+N\tau$ for $\tau$, we obtain
\begin{align*}
G_{N,k,(a,b)}(c+N\tau)(N\tau)^{k} &=G_{N,k,(a',b')}\bigg(\frac{c(c+N\tau)-(c^2+1)}{c+N\tau-c}\bigg)
\\ &
=G_{N,k,(a',b')}\bigg(c-N\frac{1}{N^2\tau}\bigg)
\end{align*}
meaning that
$$E_{a,b,c}(\tau)(N\tau)^{k}=E_{a',b',c}\bigg(-\frac{1}{N^2\tau}\bigg).$$
Using the first Fourier expansion \eqref{fourier1} for $E_{a',b',c}$, we obtain
\begin{align*}
E_{a,b,c}^{N,k}(\tau)(N\tau)^{k} &=\beta_{k}\alpha^{N,k}_{a,b}
\\ &\quad
+\sum_{\substack{m,n\ge1\\m\equiv a' \bmod N}}n^{k-1} \zeta_{N}^{b'n+a'cn}\exp\bigg(-\frac{-2\pi mn i}{N^2\tau}\bigg)
\\ &\quad
+(-1)^k\sum_{\substack{m,n\ge1\\m\equiv -a' \bmod N}}n^{k-1} \zeta_{N}^{-b'n-a'cn}\exp\bigg(\frac{-2\pi mn i}{N^2\tau}\bigg),
\end{align*}
which is precisely \eqref{fourier2}.
\end{proof}

More generally, we may consider series of the form
\begin{equation}
S(\tau)=\sum_{n,m\geq 1}n^{k-1}f(n)g(m)q^{mn}, \label{S sum}
\end{equation}
with $f$ and $g$ both $N$-periodic and satisfying the parity constraint
$$f(-a)g(-b)=(-1)^k f(a)g(b)\text{ for all integers }a,b.$$
One advantage of the Eisenstein series introduced here is that it allows us to represent $S(\tau)$ as a linear combination of $E_{a,b}^{N,k}(\tau)$, up to a linear combination of `constants' $\gamma_{a,b}(\tau)$, by using (the inverse of) the finite Fourier transform. Here we define the inverse finite Fourier transform $\widehat{f}$ of $f$ by
$$\widehat{f}(n)=\frac{1}{N}\sum_{a \bmod N}\zeta_N^{-an}f(a).$$
It is known (and easily checked) that this transform satisfies
$$f(n)=\sum_{a \bmod N}\zeta_N^{an}\widehat{f}(a).$$
We will use this property to establish the following result.

\begin{lemma}
\label{lem1}
For any two $N$-periodic functions satisfying
$$f(-a)g(-b)=(-1)^k f(a)g(b)\text{ for all integers }a,b,$$
the $q$-expansions of $S(\tau)$ and the Eisenstein series $$\frac{1}{2}\sum_{a,b \bmod N}\widehat{f}(b)g(a)E^{N,k}_{a,b}$$ coincide. In other words, the identity
\begin{equation}
\sum_{n,m\geq 1}n^{k-1}f(n)g(m)q^{mn}=\frac{1}{2}\sum_{a,b \bmod N}\widehat{f}(b)g(a)(E^{N,k}_{a,b}(\tau)-\gamma^{N,k}_{a,b}(\tau))
\label{double sum}
\end{equation}
takes place.
\end{lemma}

\begin{proof}
We can assume without loss of generality that $g$ and $f$ are not identically zero.
By the imposed relation it follows that there are $n$, $m$ such that
$g(m),\allowbreak g(-m),f(n),f(-n)\neq 0$. Then $$f(-a)=(-1)^k \frac{g(-m)}{g(m)}f(a)=(-1)^k \frac{g(m)}{g(-m)}f(a),$$ so $f(a)=f(-a)$ for all integers $a$ or $f(a)=-f(-a)$ for all integers $a$. By symmetry the same property also holds for $g$. By writing out the definition of $\widehat{f}$ we find that
$$\widehat{f}(-a)g(-b)=(-1)^k\widehat{f}(a)g(b)\text{ for all integers }a,b.$$

Now we use this property to prove the identity:
\begin{align*}
\sum_{n,m\geq 1}f(n)g(m)n^{k-1}q^{mn}&=
\sum_{a\bmod N}\sum_{\substack{m, n\ge1\\m\equiv a\bmod N}}f(n)g(m)n^{k-1}q^{mn}
\\ &=\sum_{a,b\bmod N}\sum_{\substack{m, n\ge1\\m\equiv a\bmod N}}g(a)\widehat{f}(b)\zeta_N^{bn}n^{k-1}q^{mn}
\\ &=\sum_{a,b\bmod N}\widehat{f}(b)g(a)\sum_{\substack{m, n\ge1\\m\equiv a\bmod N}}\zeta_N^{bn}n^{k-1}q^{mn}
\\ &=\frac{1}{2}\sum_{a,b\bmod N}\widehat{f}(b)g(a)
\\ & \quad \times
\bigg(\sum_{\substack{m, n\ge1\\m\equiv a\bmod N}}\zeta_N^{bn}n^{k-1}q^{mn}+(-1)^k\sum_{\substack{m, n\ge1\\m\equiv -a\bmod N}}\zeta_N^{-bn}n^{k-1}q^{mn}\bigg)
\\ &=\frac{1}{2}\sum_{a,b\bmod N}\widehat{f}(b)g(a)(E^{N,k}_{a,b}-\gamma^{N,k}_{a,b}). \qedhere
\end{align*}
\end{proof}

\begin{remark}
The proof does not actually use the fact that $k$ is a nonnegative integer: this lemma is valid for general Eisenstein series of integral weight $k$.
\end{remark}

\section{The \textit{L}--value at 4 for conductor 32} \label{L value at 4}
In \cite{Zu13}, the $L$-values at $2$ and $3$ of an elliptic curve of conductor $32$ are explicitly expressed as periods, and there is a general outline of how to derive such results. We will use this to compute a representation of the $L$-value at $4$ of the elliptic curve as a period.
In this section we use the Dedekind eta function
$$\eta(\tau):=q^{1/24}\prod_{m=1}^{\infty}(1-q^m)=\sum_{n=-\infty}^{\infty}(-1)^n q^{(6n+1)^2/24};$$
which possesses the modular involution
\begin{equation}
\eta(-1/\tau)=\sqrt{-i\tau}\,\eta(\tau).
\label{dedekind}
\end{equation}
We also set $\eta_k(\tau):=\eta(k\tau)$ for short.

Recall that for a conductor $32$ elliptic curve $E$, the $L$-series is known to coincide with that for
the cusp form $f(\tau)=\eta_4^2 \eta_8^2$. This will be shown to be a product of Eisenstein series.

We have the following (Lambert series) expansion:
$$\frac{\eta_8^4}{\eta_4^2}=\sum_{m\geq 1}\bigg(\frac{-4}{m}\bigg)\frac{q^m}{1-q^{2m}}=\sum_{m, n\geq 1}a(m)b(n)q^{mn},$$ where $a(m):=\left(\frac{-4}{m}\right)$ and $b(n):=n \bmod 2$ are as in \cite{Zu13}.
This expansion can be obtained using Dirichlet convolution; in this case $$a(m)=\sum_{\substack{n\mid m \\ \frac{m}{n}{\text{ odd}}}}c(n)\mu\left(\frac{m}{n}\right),$$
where $c(n)$ is the $n$-th term in the $q$-expansion on the left and $\mu$ is the M\" obius function.
Combining this with the identity $$\eta_4^2\eta_8^2=\frac{\eta_8^4}{\eta_4^2}\frac{\eta_4^4}{\eta_8^2}$$ and using the modular involution \eqref{dedekind} we obtain
\begin{equation}
f(it)=\frac{1}{2t}\sum_{m_1,n_1\geq 1}a(m_1)b(n_1)e^{-2\pi m_1n_1t}\sum_{m_2,n_2\geq 1}a(m_2)b(n_2)e^{-2\pi m_2 n_2/(32t)}. \label{f(i)}
\end{equation}

\begin{proof}[Proof of Theorem \textup{\ref{L(E,4)}}]
We apply the above identity to the $L$-value at 4:
\begin{align*}
L(E,4)&=L(f,4)=\frac{1}{6}\int_{0}^{1} f \log^3 q \frac{\d q}{q}=
-\frac{(2\pi)^4}{6}\int_{0}^{\infty}f(it)t^3\d t
\\ &=-\frac{(2\pi)^4}{2\cdot 6}\sum_{m_1,n_1,m_2,n_2\geq 1}a(m_1)b(n_1)b(m_2)a(n_2)
\\ & \quad \times \int_{0}^{\infty}\exp\bigg(-2\pi\bigg(m_1 n_1 t+\frac{m_2 n_2}{32t}\bigg)\bigg)t^2\d t.
\end{align*}
Performing the change of variable $t=\frac{n_2}{n_1}u$ in each of the integrals yields
\begin{align*}
L(E,4)&=-\frac{4}{3}\pi^4\sum_{m_1,n_1,m_2,n_2\geq 1}a(m_1)b(n_1)b(m_2)a(n_2)\frac{n_2^3}{n_1^3}
\\ & \quad \times \int_{0}^{\infty}\exp\bigg(-2\pi\bigg(m_1 n_2 u+\frac{m_2 n_1}{32u}\bigg)\bigg)u^2\d u
\\ &=-\frac{4}{3}\pi^4\int_{0}^{\infty}\sum_{m_1,n_2\geq 1}n_2^3 a(m_1)a(n_2)\exp(-2\pi m_1 n_2 u)
\\ & \quad \times \sum_{m_2,n_1\geq 1}\frac{b(m_2)b(n_1)}{n_1^3}\exp\bigg(\frac{-2\pi m_2 n_1}{32u}\bigg)u^2\d u.
\end{align*}
Now we make another change of variable $v=\frac{1}{32u}$, so that $$u^2 \d u=-\frac{\d v}{2^{15}v^4}.$$
We have
\begin{align*}
L(E,4)&=-\frac{\pi^4}{3\cdot 2^{13}}\int_{0}^{\infty}\sum_{m_1,n_2\geq 1}n_2^3 a(m_1)a(n_2)\exp\bigg(\frac{-2\pi m_1 n_2}{32v}\bigg)
\\ & \quad \times \sum_{m_2,n_1\geq 1}\frac{b(m_2)b(n_1)}{n_1^3}\exp(-2\pi m_2 n_1 v)\frac{\d v}{v^4}
\\ &=-\frac{\pi^4}{3\cdot 2^{13}}\int_{0}^{\infty}\cE_1(i/32v)(\delta^{-3}\cE_2)(iv)\frac{\d v}{v^4},
\end{align*}
where
\label{F1}
$$\cE_1(\tau):=\sum_{n,m\geq 1}\bigg(\frac{-4}{mn}\bigg)n^3 q^{mn}$$
and
$$
(\delta^{-3}\cE_2)(\tau):=\sum_{m,n\geq 1}\frac{b(m)b(n)}{n^3}q^{mn}=\sum_{\substack{n\geq 1 \\ n\text{ odd}}}\frac{q^n}{n^3(1-q^{2n})}.
$$
We now proceed by expressing $\cE_1(-1/(32\tau))$ as an eta quotient and $(\delta^{-3}\cE_2)(\tau)$ as an eta quotient multiplied with the composition of an eta quotient and a hypergeometric series.
We have
$$\cE_1(\tau)=
\frac{\eta_4^{16}}{\eta_2^4 \eta_8^4}-32\frac{\eta_2^4\eta_8^{12}}{\eta_4^8},$$
thus
\begin{align*}
v^{-4}\cE_1(\tau)|_{\tau=-1/(32v)}
&=\biggl(2^{12}\frac{\eta_8^{16}}{\eta_4^4\eta_{16}^4}-32\cdot2^8\frac{\eta_4^{12}\eta_{16}^4}{\eta_8^8}\biggr)\bigg|_{\tau=v}
\\
&=-2^{12}\cE_1(2\tau)|_{\tau=v}.
\end{align*}

If we define $\tilde{x}(\tau)=4\eta_8^{4}/\eta_2^4$ and then also the modular function
$$
X(\tau)=\tilde x(\tau)\cdot(1+\tilde x(\tau)^2)^{1/2}
=\frac{4\eta_4^{12}}{\eta_2^{12}},
$$
then we have, by Duke's formula \cite[eq.~(2$\cdot$6)]{Du08},
$$
(\delta^{-3}\cE_2)(\tau)=\frac{X(\tau)\cdot H(-4X(\tau)^2)}{4\Lambda(\tau)},
$$
where $$H(z)={}_4F_3\biggl(\begin{matrix} 1, \, 1, \, 1, \, 1 \\ \frac32, \, \frac32, \, \frac32 \end{matrix}\biggm|z\biggr)$$
and 
$$
\Lambda(\tau)
={}_3F_2\biggl(\begin{matrix} \frac12, \, \frac12, \, \frac12 \\ 1, \, 1 \end{matrix}\biggm|-4X(\tau)^2\biggr)
={}_2F_1\biggl(\begin{matrix} \frac12, \, \frac12 \\ 1 \end{matrix}\biggm|-\tilde x(\tau)^2\biggr)^2
=\frac{\eta_2^8}{\eta_4^4}
$$
by Clausen's formula \cite[eq.~(2.5.7)]{Sl66}.
Now, using this, we can write
\begin{align*}
L(E,4)&=\frac{\pi^4}{6}\int_{0}^{\infty}\cE_1(2\tau)
(\delta^{-3}\cE_2)(\tau)|_{\tau=iv}\d v
\\
&=\frac{\pi^4}{24}\int_{0}^{\infty}\bigg(\frac{\eta_8^{16}}{\eta_4^4 \eta_{16}^4}-32\frac{\eta_4^4\eta_{16}^{12}}{\eta_8^8}\bigg)\frac{\eta_4^{16}}{\eta_2^{20}}H(-4X(\tau)^2)|_{\tau=iv}\d v.
\end{align*}

If $x(\tau)=\frac{4\eta_2^4 \eta_8^8}{\eta_4^{12}}$, then 
$$\d x=\frac{4\eta_2^{12} \eta_8^8}{\eta_4^{16}}\frac{\d q}{q}=\frac{4\eta_2^{12} \eta_8^8}{\eta_4^{16}}\,2\pi i\d\tau.$$
Also note that we have $\tilde{x}=x/\sqrt{1-x^2}$.
\begin{remark}
The modular function $x$ satisfies $x(\tau)=z(2\tau)$, where $z=\frac{4\eta_1^4 \eta_4^8}{\eta_2^{12}}$ is the modular function we use later when computing the $L$-values for higher weights.
\end{remark}
By the change of variable, the differential form $2\pi \cE_1(2iv)(\delta^{-3}\cE_2)(iv)\d v$ transforms into
\begin{align*}
2\pi i \cE_1(2\tau)(\delta^{-3}\cE_2)(\tau)\d \tau&=\frac{1}{4}\bigg(\frac{\eta_8^{16}}{\eta_4^4 \eta_{16}^4}-32\frac{\eta_4^4\eta_{16}^{12}}{\eta_8^8}\bigg) \frac{\eta_4^{16}}{\eta_2^{20}}\times \frac{\eta_4^{16}}{\eta_2^{12}\eta_8^8} H(-4X^2)\d x
\\&=
\frac{x^4-2(1-\sqrt{1-x^2})^4}{128(1-x^2)^{11/4}(1-\sqrt{1-x^2})}\times H(-4X^2)\d x,
\end{align*}
where we use $$\bigg(\frac{\eta_8}{\eta_2}\bigg)^8=\frac{\tilde{x}^2}{16}=\frac{x^2}{16(1-x^2)},$$
$$\bigg(\frac{\eta_4}{\eta_2}\bigg)^4=\bigg(\frac{x}{4(1-x^2)}\bigg)^{1/3},$$
and
$$\bigg(\frac{\eta_{16}}{\eta_2}\bigg)^4=(1-\sqrt{1-x^2})\bigg(\frac{x}{2^{11}(1-x^2)^{7/4}}\bigg)^{1/3}$$
(since $4(\frac{\eta_4}{\eta_2})^{12}\cdot x=\tilde{x}^2$ and $1-\sqrt{1-x^2}=8\frac{\eta_2^4 \eta_8^2 \eta_{16}^4}{\eta_4^{10}}$).

When $\tau$ goes from $i\infty$ to $0$, the modular function $x(\tau)$ ranges from $0$ to $1$.
Taking $y=\sqrt{1-x^2}$, so that $$X=\frac{x}{1-x^2}=\frac{\sqrt{1-y^2}}{y^2}$$ we obtain
\begin{align*}
L(E,4)&=\frac{\pi^3}{1536}\int_{0}^{1} \frac{x^4-2(1-\sqrt{1-x^2})^4}{(1-x^2)^{11/4}(1-\sqrt{1-x^2})}\times H(-4X^2)\d x
\\ &=\frac{\pi^3}{1536}\int_{0}^{1} \frac{(1-6y+y^2)\sqrt{1-y}}{\sqrt{y^9(1+y)}}H\left(\frac{4y^2-4}{y^4}\right)\d y
\\ &
=\frac{\pi^3}{1536}\int_{0}^{1} \frac{(1-6y+y^2)\sqrt{1-y}}{\sqrt{y(1+y)}}
\iiint\limits_{[0,1]^3}\frac{\d y \d y_1 \d y_2 \d y_3}{y^4+4(1-y^2)(1-y_1^2)(1-y_2^2)(1-y_3^2)},
\end{align*}
where the integral representation for ${}_4 F_3$ from \cite[\S\,4.1]{Sl66} was used.
\end{proof}

\section{General \textit{L}--values for conductor 32 in terms of Eisenstein series}
\label{L value at k}
In this section, we show how $L(E,k)$ can be written as an integral of Eisenstein series.
In order to carry out this computation, we define the partial Fourier transform (consistent with the finite Fourier transform given earlier) $\widetilde{E}_{d,b}=\widetilde{E}_{d,b}^{N,k}$ of an Eisenstein series:
$$\widetilde{E}_{d,b}^{N,k}=\sum_{a\bmod N}\zeta_N^{da}E_{a,b}^{N,k}.$$
Note that we also have the inverse transform
$$E_{a,b}^{N,k}=\frac{1}{N}\sum_{d\bmod N}\zeta_N^{-da}\widetilde{E}_{d,b}^{N,k}.$$
The functions $\widetilde{E}_{d,b}$ have a simpler series representation, which will help us expressing modular forms in terms of the Eisenstein series.

\begin{lemma}
For any $N$, $k$, $d$ and $b$ the following holds:
\begin{align}
\widetilde{E}_{d,b}^{N,k}(\tau)=\widetilde{\gamma}_{d,b}(\tau)+\sum_{n,m\geq 1}\zeta_N^{dm+bn}n^{k-1}q^{mn}
+(-1)^{k}\sum_{n,m \geq 1}\zeta_N^{-dm-bn}n^{k-1}q^{mn};
\label{partial Fourier}
\end{align}
where
$$\widetilde{\gamma}_{d,b}(\tau)=\beta_{k}\sump_{m\equiv b\bmod N}m^{-k},$$
unless both $k=2$ and $d\equiv 0\bmod N$, in which case
$$\widetilde{\gamma}_{d,b}(\tau)=\beta_{2}\left(\sump_{m\equiv b\bmod N}m^{-2}-\frac{2\pi i}{N^2(\tau-\overline{\tau})}\right).$$
\end{lemma}
This expansion can be found by writing out the $E_{a,b}^{N,k}$ in the sum and gathering terms.
Using \eqref{double sum} allows us to write $S$ in \eqref{S sum} in the following way:
$$S(\tau)=\frac{1}{2}\sum_{a,b\bmod N}\widetilde{f}(a)\widetilde{g}(b)(\widetilde{E}_{a,b}(\tau)-\widetilde{\gamma}_{a,b}).$$

Applying the machinery developed and a formula used in \cite{Ro13}, we can now express $L(E,k)$ in terms of Eisenstein series:
\begin{lemma}
The $L$-value of an elliptic curve $E$ of conductor $32$ at $k\geq 2$ equals
	
\begin{align}
L(E,k)
&=\frac{8\pi^{2k-1}i^k}{(k-1)!(k-2)!}\int_0^\infty \int_v^\infty (z-v)^{k-2}\cE_1(2iv) \cE_2(iz)\d z \d v,
\label{intermediate L}
\end{align}
where $\cE_1$ and $\cE_2$ are Eisenstein series of weight $k$. Explicitly, for even $k$:
$$\cE_1=\tfrac{i}{2}(E_{1,-1}^{4,k}-E_{1,1}^{4,k})\quad \text{and}\quad
\cE_2=\tfrac{1}{4}(E_{1,0}^{2,k}-E_{1,1}^{2,k}),$$
and for odd $k$:
$$\cE_1=\tfrac{i}{2}(E_{1,1}^{4,k}+E_{1,-1}^{4,k}) \quad \text{and} \quad
\cE_2=\tfrac{i}{4}(\widetilde{E}_{1,2}^{4,k}-\widetilde{E}_{1,0}^{4,k})=\tfrac{i}{4}\sum_{a\bmod 4}i^{-a}(E_{a,0}^{4,k}-E_{a,2}^{4,k}).$$
\end{lemma}
\begin{proof}
We split up the proof in two cases, depending on the parity of $k$. The methods used in both cases are the same, but the coefficients $a(n)$ and $b(n)$ in \eqref{f(i)} are paired differently, so that different Eisenstein series get involved.

\medskip
\noindent
(a) \textsl{Case of $k$ even}.
We have
\begin{align*}
L(E,k)&=-\frac{(2\pi)^k}{(k-1)!}\int_{0}^{\infty}f(it)t^{k-1}\d t
\\ &
=-\frac{(2\pi)^k}{2(k-1)!}\sum_{m_1,m_2,n_1,n_2\geq 1}a(m_1)b(n_1)b(m_2)a(n_2)
\\ & \quad
\times\int_{0}^{\infty}\exp\bigg(-2\pi\bigg(m_1 n_1 t+\frac{m_2 n_2}{32t}\bigg)\bigg)t^{k-2}\d t,
\end{align*}
and the change of variable $t=\frac{n_2}{n_1}u$ applied to each individual sum gives
\begin{align*}
L(E,k)&=-\frac{(2\pi)^k}{2(k-1)!}\int_{0}^{\infty} \sum_{m_1,n_2\geq 1}n_2^{k-1}a(m_1)a(n_2)\exp(-2\pi m_1 n_2 u)
\\ & \quad
\times \sum_{m_2,n_1\geq 1}\frac{b(m_2)b(n_1)}{n_1^{k-1}}\exp\bigg(\frac{-2\pi m_2 n_1}{32u}\bigg)u^{k-2}\d u.
\end{align*}
Now take $v=\frac{1}{32u}$, hence $u^{k-2}\d u=-\frac{1}{32^{k-1}}\frac{\d v}{v^k}$ and
\begin{align*}
L(E,k)&=-\frac{(2\pi)^k}{2\cdot 32^{k-1}(k-1)!} \sum_{m_1,n_2\geq 1}n_2^{k-1}a(m_1)a(n_2)\exp\bigg(\frac{-2\pi m_1 n_2}{32v}\bigg)
\\ & \quad
\times \sum_{m_2,n_1\geq 1}\frac{b(m_2)b(n_1)}{n_1^{k-1}}\exp(-2\pi m_2 n_1 v)\frac{\d v}{v^k}.
\end{align*}
Define
$$\cE_1(\tau):=\sum_{m,n\geq 1}a(m)a(n) n^{k-1}q^{mn}
=\sum_{m,n\geq 1}\bigg(\frac{-4}{mn}\bigg) n^{k-1}q^{mn},$$
$$\cE_2(\tau):=\sum_{m,n\geq 1}b(m)b(n) n^{k-1}q^{mn}
=\sum_{\substack{m,n\geq 1 \\ m,n\text{ odd}}}n^{k-1}q^{mn},$$
so that
$$L(E,k)=-\frac{(2\pi)^k}{2\cdot 32^{k-1}(k-1)!}\int_{0}^{\infty}\cE_1\bigg(\frac{i}{32v}\bigg)\delta^{-k+1}(\cE_2)(iv)\frac{\d v}{v^k}.$$
It remains to write $\cE_1$ and $\cE_2$ in terms of Eisenstein series and apply a modular transformation to $\cE_1$. As $\big(\frac{-4}{m}\big)=\frac{i^m-i^{-m}}{2i}$, we have
$$
\cE_1=
\tfrac{1}{4}\widetilde{E}_{1,-1}^{4,k}-\tfrac{1}{4}\widetilde{E}_{1,1}^{4,k}
=\tfrac{i}{2}(E_{1,-1}^{4,k}-E_{1,1}^{4,k}),$$
and
$$
\cE_2=\tfrac{1}{8}(\widetilde{E}_{1,1}^{2,k}-\widetilde{E}_{1,0}^{2,k}-\widetilde{E}_{0,1}^{2,k}+\widetilde{E}_{0,0}^{2,k})
=\tfrac{1}{4}(-E_{1,1}^{2,k}+E_{1,0}^{2,k}).
$$
By \eqref{mod involution} we have
\begin{align*}
\cE_1\bigg(\frac{i}{32v}\bigg)&=\frac{i}{2}\left(E_{1,-1}^{4,k}\bigg(\frac{-1}{4^2(2iv)}\bigg)-E_{1,1}^{4,k}\bigg(\frac{-1}{4^2(2iv)}\bigg)\right)
\\ &
=\frac{i}{2}(4\cdot 2iv)^k(-E_{1,-1}^{4,k}(2iv)+E_{1,1}^{4,k}(2iv))
\\ &
=-(8iv)^k \cE_1(2iv).
\end{align*}
Thus we obtain
\begin{align*}
L(E,k)&=\frac{-(2\pi)^k}{2\cdot 32^{k-1}(k-1)!}\int_{0}^{\infty}-(8iv)^k \cE_1(2iv)(\delta^{-k+1}\cE_2)(iv)\frac{\d v}{v^k}
\\ &
=\frac{16(\pi i)^k}{2^k(k-1)!}\int_{0}^{\infty}\cE_1(2iv)(\delta^{-k+1}\cE_2)(iv)\d v.
\end{align*}
Using the elementary integral identity
\begin{equation}
\frac{\exp(-2\pi a u)}{a^s}=\frac{(2\pi)^s}{\Gamma(s)}\int_u^\infty (z-u)^{s-1}\exp(-2\pi a z)\d z,
\label{elementary integral}
\end{equation}
we can write
$$\delta^{-k+1}\cE_2(iv)=\frac{(2\pi)^{k-1}}{(k-2)!}\int_v^\infty (z-v)^{k-2}\cE_2(iz)\d z.$$
Applying this to the integral for $L(E,k)$ just found we obtain equation~\eqref{intermediate L}.

\medskip
\noindent
(b) \textsl{Case of $k$ odd}.
We have
\begin{align*}
L(E,k)&=-\frac{(2\pi)^k}{(k-1)!}\int_{0}^{\infty}f(it)t^{k-1}\d t
\\ &
=-\frac{(2\pi)^k}{2(k-1)!}\sum_{m_1,m_2,n_1,n_2\geq 1}a(n_1)b(m_1)b(m_2)a(n_2)
\\ & \quad
\times\int_{0}^{\infty}\exp\bigg(-2\pi\bigg(m_1 n_1 t+\frac{m_2 n_2}{32t}\bigg)\bigg)t^{k-2}\d t,
\end{align*}
and the change $t=\frac{n_2}{n_1}u$ gives
\begin{align*}
L(E,k)&=-\frac{(2\pi)^k}{2(k-1)!}\int_{0}^{\infty} \sum_{m_1,n_2\geq 1}n_2^{k-1}b(m_1)a(n_2)\exp(-2\pi m_1 n_2 u)
\\ & \quad
\times \sum_{m_2,n_1\geq 1}\frac{b(m_2)a(n_1)}{n_1^{k-1}}\exp\bigg(\frac{-2\pi m_2 n_1}{32u}\bigg)u^{k-2}\d u.
\end{align*}
As before, take $v=\frac{1}{32u}$ to get $u^{k-2}\d u=-\frac{1}{32^{k-1}}\frac{\d v}{v^k}$ and
\begin{align*}
L(E,k)&=-\frac{(2\pi)^k}{2\cdot 32^{k-1}(k-1)!} \sum_{m_1,n_2\geq 1}n_2^{k-1}b(m_1)a(n_2)\exp\bigg(\frac{-2\pi m_1 n_2}{32v}\bigg)
\\ & \quad
\times \sum_{m_2,n_1\geq 1}\frac{b(m_2)a(n_1)}{n_1^{k-1}}\exp(-2\pi m_2 n_1 v)\frac{\d v}{v^k}.
\end{align*}
This time, define
$$\widehat{\cE}_1(\tau):=\sum_{m,n\geq 1}b(m)a(n) n^{k-1}q^{mn}
=\sum_{\substack{m,n\geq 1 \\ m\text{ odd}}}\bigg(\frac{-4}{n}\bigg) n^{k-1}q^{mn},$$
$$\cE_2(\tau):=\sum_{m,n\geq 1}b(m)a(n) m^{k-1}q^{mn}
=\sum_{\substack{m,n\geq 1 \\ m\text{ odd}}}\bigg(\frac{-4}{n}\bigg) m^{k-1}q^{mn},$$
so that
$$L(E,k)=-\frac{(2\pi)^k}{2\cdot 32^{k-1}(k-1)!}\int_{0}^{\infty}\cE_1\bigg(\frac{i}{32v}\bigg)\delta^{-k+1}(\cE_2)(iv)\frac{\d v}{v^k}.$$
Since $$b(m)a(n)=\frac{i^n-i^{-n}-i^{n+2m}+i^{-n-2m}}{4i},$$
we obtain
$$
\widehat{\cE}_1=
\frac{\widetilde{E}_{0,1}^{4,k}-\widetilde{E}_{2,1}^{4,k}}{4i}
=\frac{E_{1,1}+E_{-1,1}}{2i}
$$
and
$$
\cE_2=
\frac{\widetilde{E}_{1,0}^{4,k}-\widetilde{E}_{1,2}^{4,k}}{4i}
=\frac{1}{4i}\sum_{a\bmod 4}i^{-a}(E_{a,2}^{4,k}-E_{a,0}^{4,k}).
$$
By \eqref{mod involution} we have
\begin{align*}
\widehat{\cE}_1\bigg(\frac{i}{32v}\bigg)&=\frac{1}{2i}\left(E_{1,1}^{4,k}\bigg(\frac{-1}{4^2(2iv)}\bigg)+E_{-1,1}^{4,k}\bigg(\frac{-1}{4^2(2iv)}\bigg)\right)
\\ &
=-\frac{i}{2}(4\cdot 2iv)^k(E_{1,-1}^{4,k}(2iv)+E_{1,1}^{4,k}(2iv))
\\ &
=-(8iv)^k \cE_1(2iv),
\end{align*}
where $$\cE_1=\frac{i}{2}(E_{1,1}^{4,k}+E_{1,-1}^{4,k}).$$
Finally,
\begin{align*}
L(E,k)&=\frac{-(2\pi)^k}{2\cdot 32^{k-1}(k-1)!}\int_{0}^{\infty}-(8iv)^k \cE_1(2iv)(\delta^{-k+1}\cE_2)(iv)\frac{\d v}{v^k}
\\ &
=\frac{16(\pi i)^k}{2^k(k-1)!}\int_{0}^{\infty}\cE_1(2iv)(\delta^{-k+1}\cE_2)(iv)\d v.
\end{align*}
Applying the integral identity \eqref{elementary integral} as in the case of $k$ even, we obtain equation~\eqref{intermediate L} again.
\end{proof}

\section{From Eisenstein to hypergeometric series}
In this section, we will express $\cE_1$ and $\cE_2$ in terms of hypergeometric functions and give period representations. In order to achieve this, we first need to study two different Eisenstein series.
For even $k$, we can express $\cE_1$ and $\cE_2$ in terms of
$$
\Eis_k(\tau)= q^{1/2}\sum_{\substack{n\geq 1\\n\;\text{odd}}} \frac{n^{k-1}\bigl((-1)^{k-1}q\bigr)^{(n-1)/2}}{1+q^n};
$$
for odd $k$, we can express $\cE_2$ in terms of the above series, while $\cE_1$ in terms of
$$
\Eistilde_k(\tau)=\sum_{n\geq 1} \frac{n^{k-1}(-q)^{n}}{1+q^{2n}}.
$$

For the Eisenstein series $\cE_1$ and $\cE_2$ we have, for even $k$,
\begin{align*}
\cE_1(\tau)
&=\sum_{m,n\ge1}\biggl(\frac{-4}{mn}\biggr)n^{k-1}q^{mn},
\\
\cE_2(\tau)
&=\sum_{\substack{m,n\ge1\\m,n\;\text{odd}}}n^{k-1}q^{mn}
=\frac1i\sum_{m,n\ge1}\biggl(\frac{-4}{mn}\biggr)n^{k-1}(iq)^{mn}
=\frac{\cE_1(\tau+\frac14)}i,
\end{align*}
hence
$$\cE_1(\tau)=\Eis_k(2\tau)\quad\text{and}\quad\cE_2(\tau)=-i\Eis_k\left(2\tau+\frac{1}{2}\right)\quad\text{for $k$ even}.$$
For odd $k$, the situation is similar. We have
$$\cE_1(\tau)=i \sum_{\substack{n\geq 1 \\ n\text{ even}}} \frac{n^{k-1} (iq)^n}{1+(iq)^{2n}}=2^{k-1}i\sum_{n\geq 1} \frac{n^{k-1}(-q^2)^{n}}{1+(q^2)^{2n}}$$ and
$$\cE_2(\tau)=\sum_{\substack{m,n\geq 1 \\ m, n\,\text{odd}}} \left(\frac{-4}{n}\right) m^{k-1} q^{mn}=-i \sum_{\substack{m,n\geq 1 \\ m, n\,\text{odd}}} i^m n^{k-1}q^{mn}$$
implying
$$\cE_1(\tau)=-2^{k-1} \Eistilde_k(2\tau)\quad\text{and}\quad\cE_2(\tau)=\Eis_k(2\tau)\quad\text{for $k$ odd}.$$

For every weight, both series $\Eis_k$ and $\Eistilde_k$ can be expressed in terms of the modular functions $z, \tilde z$ and the weight $1$ modular form $F$, where
\begin{align*}
z&=\frac{4\eta_1^4\eta_4^8}{\eta_2^{12}}=4q^{1/2}(1-4q+14q^2-40q^3+\dotsb),
\\
\tilde z&=\sqrt{1-z^2}=\frac{\eta_1^8\eta_4^4}{\eta_2^{12}}=1-8q+32q^2-96q^3+\dotsb,
\\
F&={}_2F_1(\tfrac12,\tfrac12;1;z^2)=\frac{\eta_2^{10}}{\eta_1^4\eta_4^4}=1+4q+4q^2+\dotsb
=\biggl(1+2\sum_{n=1}^\infty q^{n^2}\biggr)^2=\theta(\tau)^2.
\end{align*}
(While we will not use this fact, note that $z^2$ and $\tilde{z}^2$ are modular functions for $\Gamma_1(4)$ by a theorem of Newman \cite{Ne57,Ne59}, while $F$ is a modular function of weight $1$ for the same group as it is the square of the theta function.)

The next lemma shows that, if we write 
\begin{alignat*}{4}
\Eis_k(\tau)&=\tfrac14F^kz\tilde z\cdot\mu_k(z^2) \quad&& \text{for $k$ even}, \\
\Eis_k(\tau)&=\tfrac14F^kz\cdot\mu_k(z^2) \quad&& \text{for $k$ odd}, \\
\Eistilde_k(\tau)&=\tfrac{1}{4} F^k \tilde{z}\cdot \nu_k(z^2) \quad&& \text{for $k>1$ odd},
\end{alignat*}
and $\nu_1(x)=1$,
then $\mu_k(x)\in\mathbb Z[x]$ is a polynomial for $k\ge2$
and $\nu_k(x)\in \mathbb Z[x/4]$ is a polynomial for odd positive $k$.

\begin{lemma}
\label{generating functions}
The exponential generating functions
\begin{gather*}
\operatorname{SD}(t)=\sum_{\substack{k\geq 1 \\ k\;\text{even}}}\frac{\mu_{k}(x)}{(k-1)!}\,t^{k-1},
\quad
 \operatorname{NC}(t)=\sum_{\substack{k\geq 1 \\ k\;\text{odd}}}\frac{\mu_{k}(x)}{(k-1)!}\,t^{k-1}
\\ \text{and}\quad
 \operatorname{CD}(t)=\sum_{\substack{k\geq 1 \\ k\;\text{odd}}} \frac{\nu_{k}(x)}{(k-1)!}\, t^{k-1}
\end{gather*}
can be expressed in terms of Jacobi elliptic functions:
\begin{alignat*}{4}
\operatorname{SD}(t)&=-i&&\sd(it,\sqrt{x}) &&= &&\sd(t,\sqrt{1-x}),
\\ \operatorname{NC}(t)&=&&\cn(it,\sqrt{x}) &&= &&\nc(t,\sqrt{1-x}),
\\  \operatorname{CD}(t)&=&&\nd\big(\tfrac{it}{2},\sqrt{x}\big) &&=&&\cd\big(\tfrac{t}{2},\sqrt{1-x}\bigr),
\end{alignat*}
where $\nc=\frac{1}{\cn}$, $\sd=\frac{\sn}{\dn}$, $\nd=\frac{1}{\dn}$ and $\cd=\frac{\cn}{\dn}$. Here $\sn(u,\kp),\cn(u,\kp), \dn(u,\kp)$ are the standard Jacobi elliptic functions of modulus~$\kp$.
\end{lemma}
\begin{remark}
	This lemma illustrates a more general phenomenon relating elliptic functions to generating functions of Eisenstein series. As another example, the Weierstrass function $\wp$ is connected to the ordinary Eisenstein series $E_k$. In \cite[\S\,2]{CoLa11} more of such examples can be found.
\end{remark}

Combining Lemma \ref{generating functions} with the differential equations \eqref{eq-A} and Maclaurin expansions \eqref{eq-B} we arrive at the following corollary.
\begin{corollary}
For $k>1$ the function $\mu_k(x)$ is an integer polynomial of degree at most $k$, and if $k$ is also odd then $\nu_k(x)$ is a rational polynomial of degree at most $k$ vanishing at $x=0$.
\end{corollary}
\begin{proof}[Proof of Corollary]
	From the differential equations \eqref{eq-A} we see that $\kappa=\sqrt{1-x}$ only appears raised to even powers in the Maclaurin expansion of any quotient $f(u,\kappa)=\sn(u,\kappa)^n\cn(u,\kappa)^m\dn(u,\kappa)^k$ of Jacobi elliptic functions, and the highest power of $\kappa$ appearing for the coefficient of $u^n$ is at most $\kappa^{2n}$. Thus we see that $\mu_k$ and $\nu_k$ are polynomials of degree at most $k$. Furthermore, we also see that the leading nonzero coefficient of the Maclaurin expansion of $f(u,\kappa)$ is equal to $1$, and since the derivatives of $f$ are $\mathbb{Z}[\kappa^2]$-linear combinations of quotients of the form $\sn(u,\kappa)^{n'}\cn(u,\kappa)^{m'}\dn(u,\kappa)^{k'}$ we see $\mu_k(x)\in \mathbb{Z}[x]$ and $nu_k(x)\in \mathbb{Z}[\frac{1}{2},x]$.
	$\nu_k(x)$ vanishes at $0$ directly follows from the appearance of $\kappa^2$ in the differential equation for $\dn(u,\kappa)$.
\end{proof}

\begin{remark}
Though arithmetic properties of polynomials $\nu_k(x)$ are not relevant to our derivation, we cannot refrain from reproducing some of them here together with an indication of how to prove those.
The series featured in Lemma~\ref{generating functions} are all expressible via the classical theta functions $\vt_j(z)=\vt_j(z,q)$, where $j=2,3,4$, defined in~\cite[Chap.~21]{WhWa27}.
With a help of the heat equation, their $z$-expansions can be given as follows:
$$
\vt_j(z)=\sum_{m=0}^\infty\frac{(-1)^m(2t)^{2m}}{(2m)!}\,\delta^m\vt_j
\quad\text{for}\; j=2,3,4,
$$
where $\vt_j=\vt_j(0,q)$ are the corresponding thetanulls and $\delta=q\frac{\d}{\d q}$.
In the case of the generating function $\operatorname{CD}(t)$ we obtain
\begin{align}
	\sum_{\substack{k\geq 1 \\ k\;\text{odd}}} \frac{\nu_{k}(x)}{(k-1)!}\, t^{k-1}
	&=\nd\Big(\frac{it}{2},\sqrt{x}\Big)
	= \frac{\vt_4(it/(2\vt_3^2)) / \vt_4} {\vt_3(it/(2\vt_3^2)) / \vt_3}
	\nonumber\\
	&=\bigg(\sum_{m=0}^\infty \frac{t^{2m}/\vt_3^{4m}}{(2m)!} \frac{\delta^m\vt_4}{\vt_4} \bigg)\bigg/
	\bigg(\sum_{m=0}^\infty \frac{t^{2m}/\vt_3^{4m}}{(2m)!} \frac{\delta^m\vt_3}{\vt_3} \bigg)
	\label{eq:nd}
\end{align}
with $x=\vt_2^4/\vt_3^4=1-\vt_4^4/\vt_3^4$.
Notice that the logarithmic $\delta$-derivatives of the thetanulls,
$$
\psi_j=\frac{\delta\vt_j}{\vt_j} \quad\text{for}\; j=2,3,4,
$$
satisfy Halphen's system of differential equations \cite{Ha81,Zu00}
\begin{gather*}
	\delta\psi_2=2(\psi_2\psi_3+\psi_2\psi_4-\psi_3\psi_4), \quad
	\delta\psi_3=2(\psi_2\psi_3+\psi_3\psi_4-\psi_2\psi_4), \\
	\delta\psi_4=2(\psi_2\psi_4+\psi_3\psi_4-\psi_2\psi_3).
\end{gather*}
This means that the coefficients $\rho_{m,j}=(\delta^m\vt_j)/\vt_j$ in \eqref{eq:nd} satisfy the recursion $\rho_{m,j}=\delta\rho_{m-1,j}+\psi_j\rho_{m-1,j}$ for $m=1,2,\dots$;
in particular, all such $\rho_{m,j}$ are homogeneous polynomials in $\mathbb Z[\psi_2,\psi_3,\psi_4]$ of degree~$m$.
Expanding \eqref{eq:nd} we obtain
$$
\nd\Big(\frac{it}{2},\sqrt{x}\Big)
=1+\sum_{m=1}^\infty \frac{t^{2m}/\vt_3^{4m}}{(2m)!}\cdot\rho_m
$$
for some homogeneous polynomials $\rho_m\in\mathbb Z[\psi_2,\psi_3,\psi_4]$ of degree~$m$.
On the other hand, we know that the coefficients in this $t$-expansion are all \emph{modular} (rather than quasi-modular) forms of weight $2m$, hence they are actually polynomials in
\begin{equation}
	\psi_2-\psi_4=\frac14\vt_3^4 \quad\text{and}\quad \psi_3-\psi_4=\frac14\vt_2^4,
	\label{eq:psi-diff}
\end{equation}
so that $\rho_m\in\mathbb Z[\psi_2-\psi_4,\psi_3-\psi_4]$.
Using \eqref{eq:psi-diff} we conclude that $\nu_{2m+1}=\rho_m/\vt_3^{4m}$
is a polynomial of degree $m$ in $x=\vt_2^4/\vt_3^4$ satisfying $\nu_{2m+1}(4x)\in\mathbb Z[x]$.
A similar strategy allows one to prove the related properties of the other polynomials $\mu_k(x)$.
\end{remark}

Here are examples of the polynomials for $k\leq 10$:
$$
\mu_1(x)=\mu_3(x)=1, \; \mu_5(x)=1+4x, \; \mu_7(x)=1+44x+16x^2, \; \mu_9(x)=1+408x+912x^2+64x^3,
$$
and
\begin{gather*}
\mu_2(x)=1, \; \mu_4(x)=1-2x, \; \mu_6(x)=1-16x+16x^2, \; \mu_8(x)=1-138x+408x^2-272x^3, \\
\mu_{10}(x)=1-1232x+9168x^2-15872x^3+7936x^4,
\end{gather*}
as well as
\begin{gather*}
\nu_3(x)=-\frac{x}{4}, \; \nu_5(x)=20 \left(\frac{x}{4}\right)^2-\frac{x}{4}, \; \nu_7(x)=-61\left(\frac{x}{4}\right)^3+19\left(\frac{x}{4}\right)^2-\frac{x}{4}, \\
\nu_9(x)=1385\left(\frac{x}{4}\right)^4-606\left(\frac{x}{4}\right)^3+69\left(\frac{x}{4}\right)^2-\frac{x}{4}.
\end{gather*}

\begin{proof}[Proof of Lemma~\textup{\ref{generating functions}}]
	We only spell out the details for $\operatorname{NC}(t)$, as the proofs of other identities follow the same steps.
	Writing out the definitions we have
	\begin{align*}
	\sum_{k=0}^\infty \frac{\mu_{2k+1}(z^2)}{(2k)!}t^{2k} &=\frac{4}{zF}\sum_{k=0}^\infty \frac{\Eis_{2k+1}(\tau)t^{2k}}{(2k)! F^{2k}}
	\\&=\frac{4}{zF}\sum_{k=0}^\infty \frac{t^{2k}}{(2k)!F^{2k}}\sum_{\substack{n\geq 1 \\ n\text{ odd}}}\frac{n^{2k} q^{n/2}}{1+q^n}
	\\&=\frac{4}{zF}\sum_{\substack{n\geq 1 \\ n\text{ odd}}}\left(\sum_{k=0}^\infty \frac{(nt)^{2k}}{(2k)!F^{2k}}\right)\frac{q^{n/2}}{1+q^n}
	\\&=\frac{4}{zF}\sum_{\substack{n\geq 1 \\ n\text{ odd}}}\cos\left(i\frac{nt}{F}\right)\frac{q^{n/2}}{1+q^n}
	\\&=\frac{4}{zF}\sum_{n\geq 1} \cos\left((2n+1)\frac{it}{F}\right)\frac{q^{n+1/2}}{1+q^{2n-1}}.
	\end{align*}
	Now we can use the formula for $\cn$ from  \cite[\S\,22.6]{WhWa27} (the Whittaker-Watson notation translates into $\kp=z$, $\kp'=\tilde{z}$ and $K=\frac{1}{2}\pi F$), to get the above sum equal to $\cn(it,z)$. This directly implies the first equality for $\operatorname{NC}(t)$. The second equality follows from the formulas in \cite[\S\,22.4]{WhWa27}.
\end{proof}

\begin{remark}
One alternative way of deriving that $\mu_k$ is a polynomial with the mentioned properties, is by using the continued fractions for the Laplace transforms
$$\int_0^\infty \cn(u,z)e^{-xu}\d u=\frac{4 F}{z} \sum_{\substack{n \geq 1 \\n \text{ odd}}}\frac{q^{n/2}}{(1+q^n)(x^2 F^2+n^2)}$$
found by Stieltjes in 1894 \cite[Chapitre XI]{St94} and
$$\int_0^\infty \sd(u,z)e^{-xu}\d u=\frac{4 F}{z\tilde{z}} \sum_{\substack{n \geq 1 \\n \text{ odd}}}\frac{(-1)^{\frac{n-1}{2}}q^{\frac{n}{2}}n}{(1+q^n)(x^2F^2+n^2)}$$
from one of Ramanujan's notebooks (see also \cite{Ro07}). This derivation was done by Duke in \cite[\S\,4]{Du08}.
However, it does not cover the polynomials $\nu_k$; they may be treated similarly using another continued fraction of Stieltjes from \cite{St94}.
\end{remark}

Recall that $\cF(u)$ as defined in \eqref{cF-def} has period representation \eqref{cF-alt},
the second line in its integral representation follows from the formula $$\cF(u)=u^{-1/2}\bigl(\cF(1/u)-i\cF(1-1/u)\bigr)$$ for analytic continuation of the hypergeometric function to the complex plane with the cut along the half-line $[1,\infty)$.

Combining this with the results of Lemma \ref{generating functions}, we can finally compute $L(E,k)$ as a period.

\begin{theorem}
\label{period integrals}
If $E$ is a curve of conductor $32$, then its $L$-value at an even integer $k>1$ is given as a period by
\begin{align*}
L(E,k)&=\frac{\pi^{2k-3}i^{k-2}}{2^{2k}(k-1)!(k-2)!}
\int_0^1\frac{\mu_k(\alpha^2)\,\d\alpha}{\sqrt{1-\alpha^2}}
\int_0^{2\sqrt{\alpha}/(1-\alpha)}\frac{\mu_k(-\beta^2)\,\d\beta}{\sqrt{1+\beta^2}}
\\ &\quad\times
\bigl(\cF(\alpha^2)\cF(1+\beta^2)-\tfrac12\cF(1-\alpha^2)\cF(-\beta^2) +i\cF(\alpha^2)\cF(-\beta^2)\bigr)^{k-2},
\end{align*}
and its $L$-value at an odd integer $k>1$ is given as a period by
\begin{align*}
L(E,k)&=\frac{\pi^{2k-3}i^{k+1}}{2^{k+3}(k-1)!(k-2)!}
\int_{0}^1 \frac{\nu_k(\alpha)\d \alpha}{\alpha\sqrt{1-\alpha}}
\int_{p(\alpha)}^\infty
\frac{\mu_k(\beta)\d \beta}{\sqrt{\beta}(1-\beta)}
\\ &\quad\times
\bigl(\cF(\alpha)\cF(1-\beta)-\tfrac{1}{2}\cF(1-\alpha)\cF(\beta)\bigr)^{k-2},
\end{align*}
where $$p(\alpha)=\frac{-8 \alpha+4\sqrt{4 \alpha^2+\alpha (1-\alpha)^2}}{(1-\alpha)^2}.$$
Here $\mu_k$ and $\nu_k$ are polynomials as given in Lemma \nolinebreak \textup{\ref{generating functions}}.
\end{theorem}
\begin{proof}

As before, we will consider the different parity of $k$ separately. The substitution in the following part is based on the method used in \cite[\S\,3]{Ro13}.

\medskip
\noindent
(a) \textsl{Case of $k$ even}.
Using \eqref{intermediate L} we derive
\begin{align*}
L(E,k)
&=\frac{8\pi^{2k-1}i^{k-1}}{(k-1)!(k-2)!}\int_0^\infty \Eis_k(4iv)\,\d v\int_v^\infty (z-v)^{k-2} \Eis_k(2iz+\tfrac12)\,\d z
\\
&=\frac{\pi^{2k-1}i^{k-1}}{(k-1)!(k-2)!}\int_0^\infty \Eis_k(iv)\,\d v\int_{-\frac i2+\frac12v}^{\infty} (\tfrac12w+\tfrac i4-\tfrac14v)^{k-2} \Eis_k(iw)\,\d w
\\
&=\frac{\pi^{2k-1}i^{k-1}}{4^{k-2}(k-1)!(k-2)!}\int_0^\infty \Eis_k(iv)\,\d v\int_{-\frac i2+\frac12v}^{\infty} (2w-v+i)^{k-2} \Eis_k(iw)\,\d w.
\end{align*}
Now for $Z(v)=z(iv)^2$ we use
$$
1-Z\bigl(-\tfrac i2+\tfrac12v\bigr)=\frac1{1-Z(\tfrac12v)}
$$
and
$$
(1-Z(v))^2Z(\tfrac12v)^2-16Z(v)(1-Z(\tfrac12v))=0.
$$
Passing to the new variables $\alpha=Z(v)$ and $\beta=Z(w)$, so that
$$
v=\frac{\cF(1-\alpha)}{2\cF(\alpha)}, \quad
w=\frac{\cF(1-\beta)}{2\cF(\beta)}$$
and
$$
\d v=\frac{\d\alpha}{2\pi\alpha(1-\alpha)\cF(\alpha)^2}, \quad
\d w=\frac{\d\beta}{2\pi\beta(1-\beta)\cF(\beta)^2},
$$
we arrive at
\begin{align*}
L(E,k)
&=\frac{\pi^{2k-3}i^{k-1}}{4^{k+1}(k-1)!(k-2)!}\int_0^1 \frac{\cF(\alpha)^{k-2}\mu_k(\alpha)\,\d\alpha}{\sqrt{\alpha(1-\alpha)}}
\\ &\quad\times
\int_{-4\sqrt{\alpha}/(1-\sqrt{\alpha})^2}^0
\biggl(\frac{\cF(1-\beta)}{\cF(\beta)}-\frac{\cF(1-\alpha)}{2\cF(\alpha)}+i\biggr)^{k-2}
\frac{\cF(\beta)^{k-2}\mu_k(\beta)\,\d\beta}{\sqrt{\beta(1-\beta)}}
\\ &=
\frac{\pi^{2k-3}i^{k-2}}{2^{2k}(k-1)!(k-2)!}
\int_0^1\frac{\mu_k(\alpha^2)\,\d\alpha}{\sqrt{1-\alpha^2}}
\int_0^{2\sqrt{\alpha}/(1-\alpha)}\frac{\mu_k(-\beta^2)\,\d\beta}{\sqrt{1+\beta^2}}
\\ &\quad\times
\Bigl(\cF(\alpha^2)\cF(1+\beta^2)-\frac12\cF(1-\alpha^2)\cF(-\beta^2) +i\cF(\alpha^2)\cF(-\beta^2)\Bigr)^{k-2}.
\end{align*}

\medskip
\noindent
(b) \textsl{Case of $k$ odd}.
We proceed similarly to the previous case:
\begin{align*}
L(E,k)
&=\frac{2^{k-1}8\pi^{2k-1}i^{k+1}}{(k-1)!(k-2)!}
\int_0^\infty \Eistilde_k(4iv)\,\d v\int_v^\infty (z-v)^{k-2} \Eis_k(2iz)\,\d z
\\
&=\frac{2^{k-1}\pi^{2k-1}i^{k+1}}{(k-1)!(k-2)!}
\int_{0}^\infty \Eistilde_k(iv)\,\d v\int_{\frac12v}^{\infty}
(\tfrac12w-\tfrac14v)^{k-2} \Eis_k(iw)\,\d w
\\
&=\frac{2^{k-1}\pi^{2k-1}i^{k+1}}{4^{k-2}(k-1)!(k-2)!}
\int_{0}^\infty \Eistilde_k(iv)\,\d v\int_{\frac12v}^{\infty}
(2w-v)^{k-2} \Eis_k(iw)\,\d w.
\end{align*}
We use the same substitution as in the even case.
Note that for this choice of $v$ (hence $\alpha$) we have
$$
Z\bigl(\tfrac12v\bigr)
=\frac{-8 Z(v)+4\sqrt{4 Z(v)^2+Z(v) (1-Z(v))^2}}{(1-Z(v))^2}=p(Z(v)),
$$
so that the above integral becomes
\begin{align*}
L(E,k)&=\frac{\pi^{2k-3}i^{k-1}}{2^{k+3}(k-1)!(k-2)!}
\int_{0}^1 \frac{\cF(\alpha)^{k-2}\nu_k(\alpha)\d \alpha}{\alpha\sqrt{1-\alpha}}
\int_{p(\alpha)}^\infty
\frac{\cF(\beta)^{k-2}\mu_k(\beta)\d \beta}{\sqrt{\beta}(1-\beta)}
\\ &\quad\times
\biggl(\frac{\cF(1-\beta)}{\cF(\beta)}-\frac{\cF(1-\alpha)}{2\cF(\alpha)}\biggr)^{k-2}
\\ &=
\frac{\pi^{2k-3}i^{k-1}}{2^{k+3}(k-1)!(k-2)!}
\int_{0}^1 \frac{\nu_k(\alpha)\d \alpha}{\alpha\sqrt{1-\alpha}}
\int_{p(\alpha)}^\infty
\frac{\mu_k(\beta)\d \beta}{\sqrt{\beta}(1-\beta)}
\\ &\quad\times
\left(\cF(\alpha)\cF(1-\beta)-\tfrac{1}{2}\cF(1-\alpha)\cF(\beta)\right)^{k-2}.
\qedhere
\end{align*}
\end{proof}

\begin{proof}[Proof of Theorems~\textup{\ref{th2}} and \textup{\ref{th3}}]
To deduce Theorems \ref{th2} and \ref{th3} from Theorem \ref{period integrals}, we use a formula for the Hadamard product of generating functions: If $A(x)=\sum_{n=0}^\infty a_n x^n$ and $B(x)=\sum_{n=0}^\infty b_n x^n$
are generating functions, then $$\sum_{n=0}^\infty a_n b_n x^n=\int_0^1 A(e^{2\pi i t}\sqrt{x})B(e^{-2\pi i t}\sqrt{x}) \d t \quad\text{for}\; x\ge0.$$
It follows that
$$\sum_{\substack{k\geq 1 \\ k\;\text{even}}} \frac{\mu_{k}(\alpha^2)}{(k-1)!}\frac{\mu_{k}(-\beta^2)}{(k-1)!}x^{k-2}=\frac{1}{x}\int_0^1 \sd\big(e^{2\pi it}\sqrt{x},\sqrt{1-\alpha^2}\big)\sd\big(e^{-2\pi it}\sqrt{x},\sqrt{1+\beta^2}\big)\d t$$
and
\begin{align*}
&
\sum_{\substack{k>1 \\ k\;\text{odd}}} \frac{\nu_{k}(\alpha)}{(k-1)!}\frac{\mu_{k}(\beta)}{(k-1)!}x^{k-2}
\\ &\qquad
=\frac{1}{x}\int_0^1 \big(\cd\big(\tfrac{1}{2}e^{2\pi it}\sqrt{x},\sqrt{1-\alpha}\big)\nc\big(e^{-2\pi it}\sqrt{x},\sqrt{1-\beta}\big)-1\big)\d t;
\end{align*}
combined with Theorem \ref{period integrals}, these give both Theorems \ref{th2} and \ref{th3}.
\end{proof}


\end{document}